\documentclass[preprint,11pt]{elsarticle}
 \usepackage{amsmath,amssymb,amsthm,amsfonts}
 \usepackage{epstopdf}
\usepackage{pgfpages,tikz,tikz-cd,pgfkeys,pgfplots}
\usetikzlibrary{arrows,positioning,matrix,fit,backgrounds,shapes}
 \usepackage[dvips]{epsfig}
 \usepackage{color}
\usepackage{graphicx}
\usepackage{xcolor}
\usepackage{pgf}
\usepackage{soul}
 \def\draw #1 by #2 (#3){
  \vbox to #2{
    \hrule width #1 height 0pt depth 0pt
    \vfill
    \special{picture #3} % this is the low-level interface
    }
  }

 \def\scaleddraw #1 by #2 (#3 scaled #4){{
  \dimen0=#1 \dimen1=#2
  \divide\dimen0 by 1000 \multiply\dimen0 by #4
  \divide\dimen1 by 1000 \multiply\dimen1 by #4
  \draw \dimen0 by \dimen1 (#3 scaled #4)}
  }

\newtheorem{theorem}{Theorem}[section]

\newtheorem{problem}[theorem]{Problem}

\newtheorem{lemma}[theorem]{Lemma}
\newtheorem{corollary}[theorem]{Corollary}

\newtheorem{nt}{Note}

\theoremstyle{definition}%
\newtheorem{defin}[theorem]{Definition}

\newtheorem{example}[theorem]{Example}
 \newtheorem{proposition}[theorem]{Proposition}
 \newtheorem{rule-def}[theorem]{Rule}

\journal{Discrete Mathematics}

\begin{document}

\begin{frontmatter}

\newcommand{\la}{\lambda}
 \newcommand{\si}{\sigma}
 \newcommand{\ol}{1-\lambda}
 \newcommand{\be}{\begin{equation}}
 \newcommand{\ee}{\end{equation}}
 \newcommand{\bea}{\begin{eqnarray}}
 \newcommand{\eea}{\end{eqnarray}}

\title{Structural properties of Toeplitz graphs}

\author[label1]{Seyed Ahmad Mojallal}
\ead{ahmad\_mojalal@yahoo.com}
\author[label4]{Ji-Hwan Jung}
\ead{jihwanjung@snu.ac.kr}
\author[label2]{Gi-Sang Cheon}
\ead{gscheon@skku.edu}
\author[label2,label3]{Suh-Ryung Kim}
\ead{srkim@snu.ac.kr}
\author[label2]{Bumtle Kang\corref{cor1}}
\ead{lokbt1@skku.edu}
\cortext[cor1]{Corresponding author}

\address[label1]{Department of Mathematics and Statistics, University of Regina, Regina, Saskatchewan, S4S0A2, Canada}
\address[label2]{Applied Algebra and Optimization Research Center, Department of Mathematics, Sungkyunkwan University, Suwon $16419$, Republic of Korea}
\address[label3]{Department of Mathematics Education, Seoul National University, Seoul $08826$, Republic of Korea}
\address[label4]{Center for Educational Research, Seoul National University, Seoul $08826$, Republic of Korea}

\begin{abstract}
In this paper, we study structural properties of Toeplitz graphs.  We characterize $K_q$-free Toeplitz graphs for an integer $q \ge 3$ and give equivalent conditions for a Toeplitz graph $G_n\langle t_1, t_2,\ldots, t_k\rangle$ with  $t_1<\cdots<t_k$ and $n \ge t_{k-1}+t_{k}$ being chordal and equivalent conditions for a Toeplitz graph $G_n\langle t_1,t_2 \rangle$ being perfect. Then we compute the edge clique cover number and the vertex clique cover number of a chordal Toeplitz graph. Finally, we characterize the degree sequence $(d_1,d_2,\ldots,d_n)$ of a Toeplitz graph with $n$ vertices and show that a Toeplitz graph is a regular graph if and only if it is a circulant graph.
\end{abstract}


\end{frontmatter}



 \section{Introduction}

 An $n\times n$ matrix $T=(t_{ij})_{1\le i,j\le n}$ is called a {\it Toeplitz matrix} if $t_{i,j}=t_{i+1,j+1}$ for each $\{i,j\} \subset [n-1]$ where $[m]$ denotes the set $\{1,2,\ldots,m\}$ for a positive integer $m$.
 Toeplitz matrices are precisely those matrices that are constant along all diagonals parallel to the main diagonal, and thus a Toeplitz matrix is
 determined by its first row and column. Toeplitz matrices occur in a large variety of areas in pure and applied mathematics. For example,
 they often appear when differential or integral equations are discretized and arise in physical data-processing applications and in the theories of orthogonal polynomials, stationary processes, and moment problems; see Heinig and Rost \cite{HR}. Other references on
 Toeplitz matrices are Gohberg \cite{Goh} and lohvidov \cite{Ioh}.

   A {\it Toeplitz graph} is defined to be a simple, undirected graph whose adjacency matrix is a $(0,1)$-symmetric Toeplitz matrix. Any Toeplitz matrix mentioned in this paper has the main diagonal entries $0$.
   One can see that the first row of a symmetric Toeplitz matrix determines a unique Toeplitz graph.
 In this vein, we denote a Toeplitz graph with $n$ vertices by $G_n\langle t_1, t_2,\ldots, t_k\rangle$
 if the $1$'s in the first row of its adjacency matrix are placed at positions $1+t_1$, $1+t_2$, \ldots, $1+t_k$ with $1\le t_1<t_2<\ldots <t_k<n$. 
  In addition, we label the vertices of $G_n\langle t_1, t_2,\ldots, t_k\rangle$ with $1,\ldots,n$ so that the $i$th row of its adjacency matrix corresponds to the vertex labeled $i$. See Figure~\ref{fig:toe1} for an illustration.

  \begin{figure}
 \begin{minipage}[c]{.5\textwidth}
        \[
 \left(\begin{array}{ccccc}
0 & 1 & 1 & 0 & 1 \\
1 & 0 & 1 & 1 & 0 \\
1 & 1 & 0 & 1 & 1 \\
0 & 1 & 1 & 0 & 1 \\
1 & 0 & 1 & 1 & 0 \end{array}\right)\]
\end{minipage}
\begin{minipage}[c]{.45\textwidth}
\begin{tikzpicture}\node[anchor=south west,inner sep=0] (image) at (0,0)
{\includegraphics[width=0.6\textwidth]{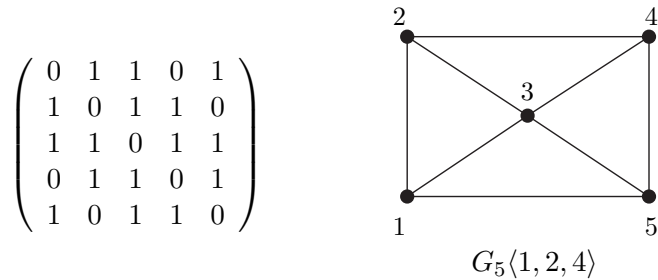}};
 \begin{scope}
 \draw (0,-0.3) node{\small $1$};
\draw (0,2.5) node{\small $2$};
\draw (1.7,1.5) node{\small $3$};
\draw (3.35,2.5) node{\small $4$};
\draw (3.35,-0.3) node{\small $5$};
\draw (1.8,-0.8) node{$G_5\langle 1,2,4 \rangle$};
    \end{scope}
\end{tikzpicture}
\end{minipage}
\caption{A $(0,1)$-symmetric Toeplitz matrix and its Toeplitz graph}\label{fig:toe1}
\end{figure}

 For $V=\mathbb{N}$ and $k<\infty$, {\it infinite Toeplitz graphs} $G_{\infty}\langle t_1, t_2,\ldots, t_k\rangle$ are defined the same way. Both types may be
 studied as special subgraphs of {\it integer distance graphs} \cite{CCH, KK, KM}.

Toeplitz graphs have been introduced by G. Sierksma and first been investigated with respect to hamiltonicity by van Dal et al.
 \cite{DTT} (see also Heuberger \cite{Heu}, Malik and Qureshi \cite{MQ}, Malik and Zamfirescu \cite{MZ} for more recent works). Infinite, bipartite Toeplitz
 graphs have been fully characterized in terms of bases and circuits by Euler et al. \cite{ELZ} (with results on the finite case presented in
 Euler \cite{Eul}). Colouring aspects are especially treated in Heuberger \cite{Heu2}, Kemnitz and Marangio \cite{KM}, Nicoloso and Pietropaoli \cite{NP}.
 Infinite, planar Toeplitz graphs have been fully characterized in Euler \cite{Eul3} providing, in particular, a complete description of the class of 3-colourable such
 graphs. Finite planar Toeplitz graphs are studied in \cite{EZ}.

 A {\it hole} is a chordless cycle of length at least $4$ as an induced subgraph, while an {\it anti-hole} is the complement of a hole. An odd hole (respectively odd anti-hole) is a hole (respectively anti-hole) with an odd number of vertices.
 A {\it chordal graph} is a simple graph without holes.
  A graph $G=(V,E)$ is an {\it interval graph} if it captures the intersection relation for some set of intervals on the real line. Formally, $G$ is an interval graph provided that one can assign to each
 $v\in V$ an interval $I_v$ such that $I_u \cap I_v$ is nonempty precisely when $uv \in  E$. Three independent vertices form an {\em asteroidal triple} in a graph $G$ if, for each two, there exists
 a path containing those two but no neighbor of the third. It is well-known that a graph $G$ is an interval graph if and only if it is chordal and has no asteroidal triple \cite{LB}.

A {\it clique} is a complete subgraph or a subset of vertices of an undirected graph such that every two distinct vertices in the clique are adjacent.
A {\it clique cover} of $G$ is a set of cliques of $G$ such that every vertex is in at least one of them. The {\it clique cover number} is the minimum size of a clique cover, and is denoted by $\theta_v(G)$.
 An {\it edge clique cover} of $G$ is a set of cliques of $G$, which together contain each edge of $G$ at least once. The smallest cardinality of any edge clique cover of $G$ is called the {\it edge clique cover number} of $G$, and is denoted by $\theta_E(G)$. Those numbers exist as the vertex set (resp. the edge set) of $G$ forms a clique cover (resp. an edge clique cover) for $G$.

  The {\it chromatic number} of a graph $G$, denoted by $\chi(G)$, is the smallest number of colors needed to color the vertices of $G$ so that no two adjacent vertices share the same color.

 The clique cover number of $G$ equals the chromatic number of its complement $\overline{G}$, that is,
 \begin{equation}\label{equ3}
 \theta_v(G)=\chi(\overline{G}).
 \end{equation}

 A {\it perfect graph} is a graph $G$ such that for every induced subgraph $H$ of $G$, the clique number equals the chromatic number, i.e.,
 $\omega(H)=\chi(H)$.  A graph for which  $\omega(G)=\chi(G)$ (without any requirement that this condition also hold on induced subgraphs) is called a {\it weakly perfect graph}.
 All perfect graphs are therefore weakly perfect by definition.

 A {\it circulant matrix} $C_n$ is an $n\times n$ Toeplitz matrix in the following form:
 $$C_n=\left[
     \begin{array}{ccccc}
       c_0 & c_{n-1} & \ldots & c_2 & c_1 \\
       c_1 & c_0 & c_{n-1} &  & c_2 \\
       \vdots & c_1 & c_0 & \ddots & \vdots \\
       c_{n-2} &  & \ddots & \ddots & c_{n-1} \\
       c_{n-1} & c_{n-2} & \ldots & c_1 & c_0 \\
     \end{array}
   \right].$$
 A graph is said to be a {\it circulant graph} if it is isomorphic to a Toeplitz graph whose adjacency matrix is a $(0,1)$-symmetric circulant matrix $C_n$ where \begin{equation} \label{eq:circulant}
 c_i=c_{n-i}  \in \{0,1\}
 \end{equation}
 for each  $i \in \{1,\ldots, \left\lfloor n/2\right\rfloor\}$ and  $c_0=0$.  Circulant graphs are well-studied (see~\cite{IB, KLP, Muz, Mei, ZZ} for references). The family of circulant graphs is an important subclass of Toeplitz graphs. Circulant graphs have various applications in the design of interconnection networks in parallel and distributed computing.

\begin{figure}
\begin{center}

\begin{tikzpicture}
    \node[anchor=south west,inner sep=0] (image) at (0,0)
{\includegraphics[height=0.285\textheight]{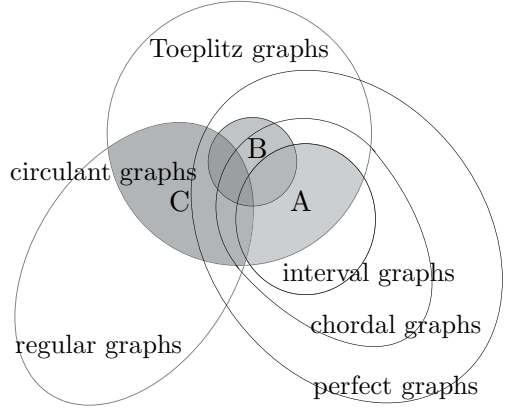}};
\begin{scope}[x={(image.south east)},y={(image.north west)}]
\node[anchor=north] at (0.51,0.925)
    {{\small Toeplitz graphs}};
\node[anchor=north] at (0.25,0.2)
    {{\small regular graphs}};
\node[anchor=north,left] at (0.45,0.57)
    {{\small circulant graphs}};
\node[anchor=north] at (0.8,0.1)
    {{\small perfect graphs}};
\node[anchor=north] at (0.8,0.25)
    {{\small chordal graphs}};
\node[anchor=north] at (0.75,0.375)
    {{\small interval graphs}};
\node[anchor=north] at (0.625,0.55)
    {A};
\node[anchor=north] at (0.544,0.680)
    {B};
\node[anchor=north] at (0.4,0.55)
    {C};
    \end{scope}
\end{tikzpicture}
\caption{Where do Toeplitz graphs stand?} \label{fig:hierachy}
\end{center}
\end{figure}

 The paper is organized as follows. In Section~\ref{sec1}, we characterize $K_q$-free Toeplitz graphs for an integer $q \ge 3$, where $K_q$ denotes the complete graph with $q$ vertices, and then we compute the edge clique cover number and the vertex clique cover number of a Toeplitz graph. In section~\ref{sec2}, we study holes in Toeplitz graphs and give a condition for a Toeplitz graph not having holes, which leads to a characterization of chordal Toeplitz graphs. Then we give equivalent conditions for a Toeplitz graph $G=G_n\langle t_1,t_2 \rangle$ being perfect. In Section~\ref{sec:degree}, we characterize the degree sequence $(d_1,d_2,\ldots,d_n)$ of a Toeplitz graph with $n$ vertices and show that a Toeplitz graph is a regular graph if and only if it is a circulant graph.
  In Section~\ref{sec:open}, we propose open problems.

  Our results are summarized in Figure~\ref{fig:hierachy}. The grey region $A$ represents the set of Toeplitz graphs $G=G_n\langle t_1, t_2,\ldots, t_k\rangle$ with $t_1<\cdots<t_k$ and $n \ge t_{k-1}+t_{k}$ being chordal for a fixed positive integer $n$ (Theorem~\ref{thm:chordequiv}); the region $B$ represents the set of  odd-hole-free Toeplitz graphs $G=G_n\langle t_1,t_2 \rangle$ (Theorem~\ref{theorem:perfectgraph}); the region $C$ represents the set of circulant graphs (Theorem~\ref{thm16}).

 \section{Cliques in Toeplitz graphs}\label{sec1}

  In this section, we give an upper bound for the clique number of Toeplitz graphs and characterize $K_q$-free Toeplitz graphs for an integer $q \ge 3$, and then we compute the edge clique cover number and the vertex clique cover number of a Toeplitz graph  $G_n \langle t,2t,\ldots,kt \rangle$. The following two results immediately follow from the definition of Toeplitz graph.

\begin{proposition}\label{prop:adj} For a positive integer $k$, let $G=G_n\langle t_1,t_2,\ldots,t_k \rangle$. Then for each $i,j \in [n]$, $i$ and $j$ are adjacent if and only if $|i-j| \in \{t_1,\ldots,t_k\}$. \end{proposition}

\begin{lemma}\label{lem:path} For positive integers $t$ and $k$, let $G=G_n\langle t, 2t,\ldots, kt\rangle$. Then $u$ and $v$ are connected  in $G$ if and only if $v-u$ is a multiple of $t$.
\end{lemma}

\begin{proposition}\label{Toeplitz-chordal-2}  For positive integers $t$ and $k$, let $G=G_n\langle t, 2t,\ldots, kt\rangle$. Then $G$ has $t$ components.
In particular, if $H_1, \ldots, H_t$ are the components of $G$, then
  $H_i$ is isomorphic to the graph $G_{\lfloor(n-i)/t\rfloor+1}\langle 1, 2,\ldots,
k\rangle$ and the vertex set of $H_i$ is  $\{s \in [n]\mid s \equiv i \pmod{t} \}$ for each $i \in [t]$.
\end{proposition}
\begin{proof}
For each $i \in [t]$, we let $V_i=\{i+st \mid s=0,1,\ldots,\left\lfloor(n-i)/t\right\rfloor\}$.
Then $|V_i| = \lfloor(n-i)/t\rfloor+1$ for each  $i \in [t]$.
By Lemma~\ref{lem:path},  $H_i:=G[V_i]$ is a component of $G$ for each  $i \in [t]$.
Fix $i \in [t]$ and let $f$ be a function from $V(G_{\lfloor(n-i)/t\rfloor+1}\langle 1, 2,\ldots,
k\rangle)$ to $V_i$ defined by $f(s+1) = i+st$ for each  $s \in \{0,1,\ldots,\lfloor(n-i)/t\rfloor\}$. It is easy to see that $f$ is a bijection.
Then  $u$ and $v$ are adjacent in $G_{\lfloor(n-i)/t\rfloor+1}\langle 1, 2,\ldots,k\rangle$ if and only if $|u-v| \in \{1,\ldots,k\}$, which is equivalent to $|(i+(u-1)t) - (i+(v-1)t)| \in \{t,\ldots, kt\}$, that is, $f(u)$ and $f(v)$ are adjacent in $H_i$.
\end{proof}

 \begin{lemma}\label{thm11} Let $G=G_n\langle t_1, t_2,\ldots, t_k\rangle$.
 Then there is a maximum clique of $G$ that contains the vertex $1$. \end{lemma}
 \begin{proof} Let $S_1=\{i_1, i_2, \ldots, i_{\ell}\}$ be a maximum clique in $G$ and let $i_1 = \min S_1$.
 Then, by Proposition~\ref{prop:adj}, $|i_u - i_v| \in \{t_1,\ldots,t_k\}$ for each  $\{u,v\} \subset [l]$.
  If $i_1 = 1$, then we are done. If $i_1 >1$, then the vertices $1, i_2-i_1+1, i_3-i_1+1, \ldots, i_{\ell}-i_1+1$ form a clique
  of size $\ell$ in $G$ by Proposition~\ref{prop:adj}. \end{proof}

 In the following, we present a condition for a Toeplitz graph being $K_q$-free.
 \begin{theorem} Let $G=G_n\langle t_1, t_2,\ldots, t_k\rangle$. Then $G$ is $K_q$-free if and only  if
 for any subset $S \subseteq [k]$ with size $q-1$, there is a pair of distinct integers $a,b \in S$ such that $|t_a - t_b| \notin \{t_1,\ldots,t_k\}$.
  \end{theorem}
\begin{proof} Let $N(1)$ be the set of neighbors of $1$ in $G$. Then $N(1) =\{t_1+1,t_2+1,\ldots, t_k+1\}$.
Now $G$ is $K_q$-free if and only if $1$ does not belong to a clique of size $q$ (by Lemma~\ref{thm11}), equivalently, any subset of $N(1)$ with size $q-1$ contains a pair $t_i+1$, $t_j+1$ such that $ |t_i-t_j|= |(t_i+1) - (t_j +1)| \notin \{t_1,\ldots,t_k\}$.

If  any subset $S \subseteq [k]$ with size $q-1$ contains a pair of distinct integers $a,b \in S$ such that $|t_a - t_b| \notin \{t_1,\ldots,t_k\}$, then any subset of $N(1)$ with size $q-1$ contains a pair $t_i+1$, $t_j+1$ such that $ |t_i-t_j| \notin \{t_1,\ldots,t_k\}$.
Suppose that any subset of $N(1)$ with size $q-1$ contains a pair $t_i+1$, $t_j+1$ such that $ |t_i-t_j| \notin \{t_1,\ldots,t_k\}$. Let $S$ be a subset of $[k]$ with size $q-1$. Then $\{ t_i +1 \mid i \in S\}$ is a subset of $N(1)$ with size $q-1$, so there is a pair $t_i+1$, $t_j+1$ such that $ |t_i-t_j| \notin \{t_1,\ldots,t_k\}$ by the last equivalence above. Thus there is a pair of distinct integers $a,b \in S$ such that $|t_a - t_b| \notin \{t_1,\ldots,t_k\}$.\end{proof}

 \begin{corollary}\label{thm12} Let $G=G_n\langle t_1, t_2,\ldots, t_k\rangle$. Then $G$ is triangle-free if and only if $|t_i-t_j|\notin \{t_1, t_2, \ldots, t_k\}$  for any pair $i,j \in [k]$.
 \end{corollary}

 For a Toeplitz graph $G:=G_n\langle t_1, \ldots, t_k\rangle$, we denote $B(G)=\{t_1, \ldots, t_k\}$.

 \begin{lemma} Let $G=G_n\langle t_1, t_2,\ldots, t_k\rangle$. Then $|t_i - t_j| \in B(G)$ for every $\{t_i,t_j\} \subset B(G)$ with $i \ne j$ if and only if $t_i = it_1$ for each $i \in [k]$.\label{lem:consecutive}\end{lemma}
 \begin{proof}
We can easily check the `if' part. To show the `only if' part, suppose that $|t_i - t_j| \in B(G)$ for every $\{t_i,t_j\} \subset B(G)$ with $i \ne j$. Then $t_2 - t_1 \in B(G)$. Since $t_2 - t_1 <
 t_2$, $t_2 - t_1 = t_1$. Therefore $t_2 = 2t_1$. By the supposition again, $t_3 - t_1 \in B(G)$. Since $t_3 -t_1 <t_3$, $t_3 - t_1 =
 t_2$ or $t_3 - t_1 = t_1$. If $t_3 - t_1 = t_1$, then $t_3 = t_2$, a contradiction. Thus $t_3 - t_1 = t_2$ and consequently $t_3 = 3t_1$.
 By repeating this procedure, we conclude that $t_i=it_1$ for each $i \in [k]$ and thus $B(G) = \{t_1, 2t_1,\ldots,kt_1\}$.
 \end{proof}

We denote the degree of the vertex $i$ by $\deg(i)$ in a Toeplitz graph.

\begin{theorem} \label{thm15} Let $G=G_n\langle t_1, t_2,\ldots, t_k\rangle$. Then $\omega(G) \le k+1$.
Furthermore, the equality holds if and only if $t_i = it_1$ for each  $i \in [k]$.\label{thm:toecliquesize}\end{theorem}
\begin{proof} By Lemma~\ref{thm11}, there is a maximum clique that contains the vertex $1$.
Then the clique is contained in the closed neighborhood of $1$. Since  $\deg(1) = k$, the inequality holds.
 Now we show the equality part.
If $B(G)=\{t_1, 2t_1, \ldots, kt_1\}$, then $\{1,1+t_1,\ldots,1+kt_1\}$ is a clique of size $k+1$ in $G$ and so $\omega(G) = k+1$.

Suppose that $\omega(G)=k+1$. Then, by Lemma~\ref{thm11}, there is a maximum clique of size $k+1$ containing $1$. Therefore $\{1,1+t_1,\ldots,1+t_k\}$ is a clique. Now, for each $i,j \in [k]$, $1+t_i$ and $1+t_j$ are adjacent since they belong to the same clique. Then $|(1+t_j) - (1+t_i)| = |t_j -t_i| \in B(G)$ by Proposition~\ref{prop:adj}, so $B(G)  = \{t_1,2t_1,\ldots,k t_1\}$ by Lemma~\ref{lem:consecutive}.\end{proof}

\begin{theorem} Let $G=G_n\langle t, 2t,\ldots, kt\rangle$. Then
$\theta_E(G) = \max \{t,n-kt\}$.
Moreover, for $H:=G_n\langle s_1, s_2,\ldots, s_k\rangle$ with $s_1 = t$ and $s_k = kt$, $\theta_E(G) \le \theta_E(H)$.
\label{thm:toecc}
\end{theorem}
\begin{proof} Suppose that $n-kt \le t$. Then $n \le (k+1)t$.
By Proposition~\ref{Toeplitz-chordal-2}, $G$ has $t$ components and so $\theta_E(G) \ge t$.
Again, by Proposition~\ref{Toeplitz-chordal-2}, each component is isomorphic to $G_{\lfloor(n-i)/t \rfloor+1}\langle 1, 2,\ldots,
k\rangle$ for each  $i\in [t]$. Since $n \le (k+1)t$, the number of vertices in a component is at most $k+1$ and so $G_{\lfloor(n-i)/t \rfloor+1}\langle 1, 2,\ldots,
k\rangle$ for each $i\in [t]$ is a complete graph by Proposition~\ref{prop:adj}. Thus $\theta_E(G) = t$.

Now, suppose that $t < n-kt$. Then $n > (k+1)t$.
Let $C_i = \{i,i+t,\ldots,i+kt\}$ for each $i \in \{1,\ldots, n-kt\}$.
Then, for each  $i \in \{1,\ldots, n-kt\}$, an element in $C_i$ is at least $i\ge 1$ and at most $i+kt \le (n-kt)+kt =n$, so $C_i$ is a vertex set of $G$ .
  In addition, $C_i$ is a clique for each  $i \in \{1,\ldots, n-kt\}$ by definition. Take an edge $uv \in E(G)$ such that $u < v$.
   If $kt < v$, then $C_{v-kt} = \{v-kt, v-(k-1)t,\ldots,v\}$ contains $u$ and $v$.
  Suppose that $v \le kt$. Then there exists an integer $r$  such that $r \equiv v \pmod{t}$ with $1 \le r \le t$. Obviously, $r \le v$ and $v-r$ is a multiple of $t$. Since $v \le kt$, $v -r \le kt$ and so $v -r \in   \{t,2t,\ldots,kt\}$. Thus $C_r = \{r,r+t,\ldots,r+kt\}$ contains $v$.
      Thus $\{C_i \mid 1 \le i \le n-kt\}$ is an edge-clique cover of $G$ and so $\theta_E(G) \le n-kt$.

  Now, let $F= \{i\ i+kt \mid 1 \le i \le n-kt\}$. Take edges $i\ i+kt$ and $j\ j+kt$ for some $1\le i < j \le n-kt$. Since $j+kt-i > kt$, $i$ and $j+kt$ are not adjacent by Proposition~\ref{prop:adj} and so $i\ i+kt$ and $j\ j+kt$ do not belong to the same clique. Thus $\theta_E(G) \ge |F| = n-kt$ and hence $\theta_E(G_n) = n-kt$.

 Let $H:=H_n\langle s_1, s_2,\ldots, s_k\rangle$ be a Toeplitz graph with $s_1 = t$ and $s_k = kt$. Note that $G=H$ if $k \le 2$, so $\theta_E(G) = \theta_E(H)$. We assume that $k \ge 3$.
Suppose that $kt <  n \le (k+1)t$. Let $I=\{i\ i+s_1 \mid 1 \le i \le s_1\}$. Since $k \ge 3$, $2s_1=2t< kt < n$ and so, by Proposition~\ref{prop:adj}, $I \subset E(H)$.
Take edges $i\ i+s_1$ and $j\ j+s_1$ for some $1\le i < j \le s_1$. Since $j-i < s_1$, $i$ and $j$ are not adjacent by Proposition~\ref{prop:adj} and so $i\ i+s_1$ and $j\ j+s_1$ do not belong to the same clique. Thus we have shown that $\theta_E(H) \ge |I| = s_1 = t = \theta_E(G)$ if $kt <  n \le (k+1)t$.
Suppose that $n > (k+1)t$. Since $s_k = kt$, $F \subset E(H)$ by Proposition~\ref{prop:adj} and so $\theta_E(H) \ge |F| =n-kt$. Therefore $\theta_E(G) \le \theta_E(H)$ if $(k+1)t < n$. Thus the ``moreover" part is true.
\end{proof}

 In the following, we compute the vertex clique cover number of a Toeplitz graph $G_n \langle t,2t,\ldots,kt \rangle$.
 \begin{theorem}
 Let $G=G_n\langle t, 2t,\ldots, kt\rangle$ for $n > (2k-1)t$. Then $$\theta_v(G)=s \left\lceil \frac{\lceil n/t\rceil}{k+1}\right\rceil+(t-s)\left\lceil \frac{\lfloor n/t\rfloor}{k+1}\right\rceil$$
 where  $s$ is the positive integer such that $s\equiv n  \pmod{t}$ and $1 \le s \le t$.
 \end{theorem}
 \begin{proof}
  Since $s\equiv n  \pmod{t}$ and $1 \le s \le t$, it follows from Proposition~\ref{Toeplitz-chordal-2} that $G$ has $t$ components, each of the first $s$ components $H_1,\ldots,H_s$ is isomorphic to $G_{\lceil n/t \rceil}\langle 1, 2,\ldots,k\rangle$, and each of the other $t-s$ components $H_{s+1},\ldots, H_{t}$ is isomorphic to $G_{\lfloor n/t \rfloor}\langle 1, 2,\ldots,k\rangle$.

   Since $H_i \simeq G_{\lceil n/t \rceil}\langle 1, 2,\ldots, k\rangle$ for each  $i \in [s]$  and $k+1$ consecutive vertices of $\{1,2,\ldots,\lceil \frac{n}{t} \rceil \}$ form a clique in $G_{\lceil n/t \rceil}\langle 1, 2,\ldots,
k\rangle$, we have $\theta_v(H_i) \le \left \lceil \frac{\lceil n/t\rceil}{k+1} \right \rceil$ for each $i \in [s]$.
Now, by Theorem~\ref{thm15}, $\omega(H_i) = k+1$ for each $i \in [s]$, so we can conclude that $\theta_v(H_i) = \left \lceil \frac{\lceil n/t\rceil}{k+1} \right \rceil$  for each $i \in [s]$.
  Similarly, we can show that $\theta_v(H_i)= \left \lceil \frac{\lfloor n/t\rfloor}{k+1} \right \rceil$ for each $i \in \{s+1,\ldots, t\}$.
    Therefore \[\theta_v(G)=\sum_{i=1}^t \theta_v(H_i)=\sum_{i=1}^s \theta_v(H_i)+\sum_{i=s+1}^t \theta_v(H_i)=s \left\lceil \frac{\lceil n/t\rceil}{k+1}\right\rceil+(t-s)\left\lceil \frac{\lfloor n/t\rfloor}{k+1}\right\rceil.\qedhere\]
  \end{proof}

\section{Chordal Toeplitz graphs and Perfect Toeplitz graphs}\label{sec2}
 In this section, we study holes in Toeplitz graphs and give a condition for a Toeplitz graph not having holes, which leads to a characterization of chordal Toeplitz graphs.
Then we give equivalent conditions for a Toeplitz graph $G=G_n\langle t_1,t_2 \rangle$ being perfect.
By Theorem~\ref{thm:toecliquesize}, we know that, for $G=G_n\langle t_1, t_2,\ldots, t_k\rangle$, $\omega(G)=k+1$ if and only if $t_i = it_1$ for each $i \in [k]$. Yet, as long as $n \ge t_{k-1}+t_k$, we add more equivalent statements such as $G$ is chordal.

\begin{proposition}\label{prop:chord} For positive integers $k$ and $t$, $G_n\langle t, 2t,\ldots,kt\rangle$ is chordal.\end{proposition}
\begin{proof} By Proposition~\ref{Toeplitz-chordal-2}, it suffices to show that
$G:=G_{m}\langle 1, 2,\ldots,k\rangle$ is chordal for each integer $m$, $m > k$.
To reach a contradiction, suppose that $G$ contains a hole $C:= v_1 v_2 \ldots v_{\ell} v_1$ for some integer $\ell \ge 4$. We identify $v_{\ell+1}$ with $v_1$. Since the sequence $C$ cannot strictly increase or decrease, either $v_{i-1}< v_i$ and $v_i >v_{i+1}$ or $v_{i-1}>v_i$ and $v_i<v_{i+1}$ for some $i$, $2\le i\le\ell$.
Assume the former. Then $v_i-v_{i-1}=|v_i-v_{i-1}|=a$ and $v_i-v_{i+1}=|v_i-v_{i+1}|=b$ for some $a,b\in\{1,2,\ldots,k\}$ by Proposition~\ref{prop:adj}.
 Since $\ell \ge 4$, $a\ne b$ and so  $|v_{i-1}-v_{i+1}|=|a-b|\in\{1,2,\ldots,k-1\}$. Therefore $v_{i-1}v_{i+1}$ is a chord of $C$ and we reach a contradiction.
 We can similarly show that $C$ also has a chord in the latter case. Thus $G$ is chordal.
\end{proof}

For a path $P = v_1 v_2 \cdots v_k$, we denote by $P^{-1}$ the path $v_k v_{k-1} \cdots v_2 v_1$.

\begin{lemma} \label{lem:k2cycle} Let $G= G_n\langle t_1, t_2\rangle$ be a Toeplitz graph with $n \ge t_1+t_2$. If $t_2 \ne 2t_1$, then $G$ has a hole of length $(t_1+t_2)/ \gcd(t_1,t_2)$. \end{lemma}
\begin{proof}
For each $i \in \{1,\ldots,t_1\}$, let $P_i$ be the path such that\[
P_i = i \ i+t_1 \ i+2t_1 \cdots i+  \left\lfloor(n-i)/t_1\right\rfloor  t_1.\]
Then $P_1,\ldots,P_{t_1}$ are $t_1$ disjoint paths which contain all the vertices of $G$.

Now we construct a cycle in the following way.
We start from the vertex $1$ and consider the edge $1\ 1+t_2$. Then $1+t_2$ is on the path $P_j$ for $j \equiv 1+t_2 \pmod{t_1}$.
We denote by $P_1'$ the $(1+t_2,j)$-section of ${P_j}^{-1}$.
Since $j \le t_1$, $j+t_2\le t_1+t_2 \le n$ and so $j+t_2$ is a vertex of $G$.
Then we take the edge $j\ j+t_2$ and the path $P_k$ where $k \equiv j+t_2 \equiv 1+2t_2 \pmod{t_1}$.
We denote by $P_2'$ the $(j+t_2,k)$-section of ${P_k}^{-1}$.
Then $1$ $P_1'$ $P_2'$ is a $(1,k)$-path.
Noting that $1 \equiv 1+ x t_2 \pmod{t_1}$ has a solution $s:=t_1/d$ where $d=\gcd(t_1,t_2)$, we may conclude that $1$ $P_1'$ $P_2'$ $\cdots$ $P_s'$ is a cycle in $G$.
 By construction, the length of this cycle is the minimum of sum of two positive integers $x$ and $y$ satisfying $yt_2 = xt_1$, that is,  $\min \{x+y \mid yt_2 - xt_1=0, x,y \in \mathbb{Z}^+\} = (t_1+t_2)/d$.

To show that the cycle has no chord, take two vertices $u$ and $v$ with $u<v$ on the cycle. Then
$u=1+u_1t_2 - v_1t_1$ and $v=1+u_2t_2 - v_2t_1$ for some $\{u_1,u_2\} \subset \{1,\ldots, t_1/d\}$ and $\{v_1,v_2\} \subset \{1,\ldots, t_2/d\}$. If $u$ and $v$ are adjacent, $v-u$ is either $t_1$ or $t_2$.
Suppose $v-u=t_1$. Then $(u_2-u_1)t_2 - (v_2-v_1+1)t_1 = 0$. However, $|u_2-u_1| < t_1/d$ and so $u_2=u_1$ and $|v_2- v_1|=1$, which implies that $u$ and $v$ are consecutive on the cycle.
Similarly one can show that if $v-u=t_2$, then $u$ and $v$ are also consecutive.
Thus the cycle has no chord.
Since $t_2 \ne 2t_1$, $(t_1+t_2)/d >3$ and so the cycle is a hole.
\end{proof}

 The following theorem characterizes the chordal Toeplitz graphs $G=G_n\langle t_1, t_2,\ldots, t_k\rangle$ with $n \ge t_{k-1}+t_k$.
\begin{theorem} Let $G=G_n\langle t_1, t_2,\ldots, t_k\rangle$ be a Toeplitz graph. If $n \ge t_{k-1}+t_k$, then the following statements are equivalent.
\begin{itemize}
\item[(i)] $G$ is interval.
\item[(ii)] $G$ is chordal.
\item[(iii)] $t_i=it_1$ for each $i\in [k]$.
\item[(iv)] $\omega(G) = k+1$.
\end{itemize}\label{thm:chordequiv}
\end{theorem}
\begin{proof}
$\mbox{(i)}\Rightarrow\mbox{(ii)}$ is obvious.
 By Theorem~\ref{thm:toecliquesize}, $\mbox{(iii)}\Leftrightarrow\mbox{(iv)}$.
 To complete the proof, we shall show that $\mbox{(ii)} \Rightarrow\mbox{(iii)}$ and $\mbox{(iii)} \Rightarrow\mbox{(i)}$.
 To show $\mbox{(ii)} \Rightarrow\mbox{(iii)}$, we denote by $C_4(t_i,t_j)$ the $4$-cycle
 \[1 \ (1+t_i) \ (1+t_i+t_j) \ (1+t_j) \  1\] for each $\{i,j\} \subset [k]$ with $i<j$ and $t_i+t_j < n$.
 Since $G$ is chordal, $C_4(t_i,t_j)$ has a chord for each $\{i,j\} \subset [k]$ with $i<j$ and $t_i+t_j<n$.
Thus, for each $\{i,j\} \subset [k]$ with $i<j$ and $t_j+t_i <n$, \begin{equation}
t_j - t_i \in \{t_1, t_2,\ldots, t_k\} \quad \mbox{or} \quad t_j+t_i \in \{t_1, t_2,\ldots, t_k\}. \label{eq:c4chord}\end{equation}

Suppose that $k=2$. If $t_2 \ne 2t_1$, then $G$ has a hole of length $(t_1+t_2)/ \gcd(t_1,t_2)$ by Lemma~\ref{lem:k2cycle} and we reach a contradiction. Thus $t_2 = 2t_1$.
Now we suppose that $k=3$. Since $t_1+t_2<t_1+t_3 < t_2+t_3 \le n$, $C_4(t_1,t_2)$ and $C_4(t_1,t_3)$ exist.
Therefore we have $t_2-t_1 = t_1$ or $t_1+t_2 =t_3$ from $C_4(t_1,t_2)$ and $t_3-t_1 \in \{t_1,t_2\}$ from $C_4(t_1,t_3)$.
If $t_2-t_1 = t_1$, then $t_3 - t_1 = t_2$, so $t_2= 2t_1$ and $t_3=3t_1$.
Assume that $t_1+t_2=t_3$. To reach a contradiction, suppose that $t_2 \ne 2t_1$. By definition, $H=G_{t_1+t_2}\langle t_1, t_2 \rangle$ is a subgraph of $G$ and $H$ contains a hole $C$ by Lemma~\ref{lem:k2cycle}. Since $G$ is chordal, $C$ has a chord in $G$. Then the difference of two ends of a chord is $t_3$, for otherwise the chord also exists in $H$. However, since $C$ is a subgraph of $H$, the difference of any pair of vertices on $C$ is at most $t_1+t_2-1 = t_3-1$ and we reach a contradiction.
Thus $t_2 = 2t_1$ and $t_3 = t_1+t_2 = 3t_1$.

 Now suppose $k\ge 4$.
  We consider the cycle $C_4(t_j,  t_k)$ for each $j \in [k-2]$.
 Then for each $j \in [k-2]$,  $t_j+t_k < n$  and $t_j+t_k \notin \{t_1,\ldots,t_k\}$ and so $t_k - t_j \in \{t_1,\ldots,t_k\}$ for each $j \in [k-2]$ by~\eqref{eq:c4chord}.
 To reach a contradiction, suppose that $t_k-t_{k-1} \notin B$. 
 Then $t_k-t_j \ne t_{k-1}$ for any $j \in [k-2]$ and so $\{t_k-t_1,t_k-t_2,\ldots,t_k-t_{k-2}\}= \{t_1,\ldots,t_{k-2}\}$. Since $t_k-t_1$ is the largest element in the set, we have $t_k-t_1 = t_{k-2}$.
  Then \[
  t_k= t_{k-2}+t_1 < t_{k-1}+t_{k-2} < t_{k}+t_{k-1} \le n,\] so $t_{k-1}-t_{k-2} \in B$ by \eqref{eq:c4chord}. However, $t_{k-1}-t_{k-2} < t_k-t_{k-2} = t_1$ 
and we reach a contradiction. Therefore $t_k-t_{k-1} \in B$ and so $\{t_k-t_1,t_k-t_2,\ldots,t_k-t_{k-1}\} = \{t_1,\ldots,t_{k-1}\}$.
 Thus  \begin{equation}
 t_k = t_j + t_{k-j} \quad \mbox{for each } j \in [k-1].\label{eq:chord1}\end{equation}
  For each  $j \in \{2,\ldots,k-2\}$,
    \[t_{k-1} - t_j \le t_{k-1} -t_2 < t_k -t_2 = t_{k-2}\] by \eqref{eq:chord1}. Yet, since $t_j+t_{k-1} > t_k$ for each  $j \in \{2,\ldots,k-2\}$ by \eqref{eq:chord1}, $t_{k-1} - t_j \in \{t_1,\ldots,t_k\}$ for each  $j \in \{2,\ldots,k-2\}$.
 Therefore \begin{equation} t_{k-1} - t_j  \le t_{k-3} \quad \mbox{for each }  j \in \{2,\ldots,k-2\}.\label{eq:k-1} \end{equation}
  In addition, since $t_{k-1}-t_{k-2} < t_{k} - t_{k-2} = t_2$ by \eqref{eq:chord1}, we have $t_{k-1}-t_{k-2} = t_1$. Thus, by \eqref{eq:k-1}, \begin{equation}
 t_{k-1} = t_j + t_{k-j-1} \quad \mbox{for each }   j \in \{1,\ldots,k-2\}.\label{eq:chord2}\end{equation}
By subtracting \eqref{eq:chord2} from \eqref{eq:chord1} for each $j \in \{1,\ldots,k-2\}$, we have\begin{equation}
t_k-t_{k-1} = t_{k-j} - t_{k-j-1}\quad \mbox{for each }  j \in \{1,\ldots,k-2\}.\label{eq:chord3}\end{equation}
By \eqref{eq:chord1}, $t_k - t_{k-1} = t_1$. Thus, by \eqref{eq:chord3},
\[ t_j = t_{j-1}+t_1 = t_{j-2}+2t_1 = \cdots = t_1 + (j-1)t_1 = jt_1\]
for each $j \in [k]$.

Now we show $\mbox{(iii)}\Rightarrow\mbox{(i)}$. Suppose that $G=G_n\langle t_1, 2t_1,\ldots, kt_1\rangle$. 

By Proposition~\ref{Toeplitz-chordal-2}, each component of $G$ is $G_{n_i} \langle 1,2,\ldots, k\rangle $ for some $n_i \le n$.
Therefore it suffices to show that $G_{n_i} \langle 1,2,\ldots, k\rangle $ is an interval graph for any $n_i$.
To each vertex $u \in [n_i]$, we assign an interval $[u, u+k]$. We note that $[u,u+k] \cap [v,v+k] \ne \emptyset$  if and only if $v-u \in [k]$, that is, $u$ and $v$ are adjacent in $G_{n_i} \langle 1,2,\ldots, k\rangle $.
Thus $G_{n_i} \langle 1,2,\ldots, k\rangle$ is an interval graph and hence we may conclude $G$ is an interval graph.  \end{proof}

 In the rest of this section, we characterize perfect Toeplitz graphs in the following by utilizing the results we have shown. We first introduce classes of well-known perfect graphs.
 \begin{theorem} {\rm \cite{Hou}}\label{thm22}
 Chordal graphs, cographs and bipartite graphs are perfect.
 \end{theorem}

The following corollary is an immediate consequence of Proposition~\ref{prop:chord} and Theorem~\ref{thm22}.

\begin{corollary}
 Let $G=G_n\langle t, 2t,\ldots, kt\rangle$. Then $\chi(G)=k+1$.
 \end{corollary}
 \begin{proof}
 By Proposition~\ref{prop:chord}, $G=G_n\langle t, 2t,\ldots,  kt\rangle$ is chordal, and thus, by Theorem~\ref{thm22}, it is a
 perfect graph. Then, by Theorem \ref{thm:chordequiv}, $\chi(G)=\omega(G)=k+1.$
 \end{proof}

  Next, we characterize perfect Toeplitz graphs $G_n\langle t_1, t_2\rangle$. To do so, we introduce the following Theorem.
  A graph $G$ is called a {\it Berge graph} if it contains neither an odd hole nor an odd anti-hole as an induced subgraph.

  \begin{theorem} {\rm (Strong Perfect Graph Theorem \cite{CRST})} \label{thm21} A graph is perfect if and only if it is Berge.
 \end{theorem}

  \begin{theorem} \label{theorem:perfectgraph} Let $G=G_n\langle t_1, t_2\rangle$ with $n \ge t_1+t_2$ and $d=\gcd(t_1,t_2)$. Then the following statements are equivalent.
  \begin{itemize}
  \item[(i)] $(t_1+t_2)/d$ is even or $(t_1+t_2)/d=3$.
  \item[(ii)] $G$ is an odd-hole-free graph.
  \item[(iii)] $G$ is a perfect graph.
  \item[(iv)] $G$ is a weakly perfect graph.
  \end{itemize}
  \end{theorem}
  \begin{proof}

 Let $k_1 = t_1/d$ and $k_2 = t_2/d$.
 We first show $\mbox{(ii)}\Rightarrow\mbox{(i)}$. Since $G$ has a hole of length $k_1+k_2$ by Lemma~\ref{lem:k2cycle}, $G$ is an odd-hole-free graph only if $k_1+k_2$ is even or $k_1+k_2=3$.

 Now we show $\mbox{(i)}\Rightarrow\mbox{(ii)}$.  Suppose that $k_1+k_2 = 3$. Since $t_1 < t_2$,  $k_1=1$ and $k_2 = 2$ and so $t_2 = 2t_1$. Therefore $G$ is chordal by Proposition~\ref{prop:chord} and thus $G$ is odd-hole-free by Theorem~\ref{thm22}.
  Suppose that $k_1+k_2$ is even. We prove by contradiction. Suppose that $G$ has an odd-hole $C$ of length $\ell$ for some positive odd integer $\ell \ge 5$.
 Let $C= v_1 v_2 \ldots v_{\ell} v_1$. Then, by Proposition~\ref{prop:adj}, $|v_{i+1}-v_i| \in \{t_1,t_2\}$ for each $i \in [\ell]$ (we identify $v_1$ with $v_{\ell+1}$).
 For each $j \in \{1,2\}$, let $a_j$ be the number of indices $i$ such that $v_{i+1}-v_i = t_j$, and let $b_j$ be the number of indices $i$ such that $v_{i}-v_{i+1} = t_j$ for  each $i \in [\ell]$.
 Then the length of $C$ is $a_1+b_1+a_2+b_2$.
 Since $v_1 = v_{\ell+1}$, \begin{align*}
 0 &= (v_{\ell+1} - v_{\ell})+(v_{\ell} - v_{\ell-1}) +\cdots + (v_2 - v_1) + (v_1 - v_{\ell+1}) \\
  &= a_1t_1+a_2t_2-b_1t_1-b_2t_2,\end{align*}
   or $(a_1-b_1)t_1 = (b_2-a_2)t_2$.
 If $a_2=b_2$, then $b_1=a_1$ and so the length of $C$ is $2(a_1+a_2)$, which is a contradiction.
Thus $a_2 \ne b_2$. Then \[\frac{k_2}{k_1} = \frac{t_2}{t_1} = \frac{a_1-b_1}{b_2-a_2}.\] Therefore $\alpha k_2 = a_1-b_1 $ and  $\alpha k_1 = b_2-a_2$ for some integer $\alpha$. Then the length of $C$ is $2b_1+2a_2+\alpha(k_1+k_2)$. Since $k_1+k_2$ is even by the supposition, we reach a contradiction. Thus $G$ is odd-hole-free. Hence we have shown that $\mbox{(i)}\Leftrightarrow\mbox{(ii)}$.

By Theorem~\ref{thm21}, $\mbox{(iii)}\Rightarrow\mbox{(ii)}$. Next, we will show $\mbox{(i)}\Rightarrow\mbox{(iii)}$.
 Suppose that $k_1+k_2$ is even or $k_1+k_2=3$. If $k_2=2k_1$, then $G$ is chordal by Proposition~\ref{prop:chord} and thus $G$ is a perfect graph by Theorem~\ref{thm22}.
 Suppose that $k_2 \ne 2k_1$. Then $k_1+k_2 \ne 3$, so $k_1+k_2$ is even.
 In addition, $G$ is not chordal by Theorem~\ref{thm:chordequiv}, so $\omega(G) \le 2$ by Theorem~\ref{thm15}. Since we have shown $\mbox{(i)}\Rightarrow\mbox{(ii)}$, $G$ is an odd-hole-free graph.
  By Theorem~\ref{thm21}, it remains to show that $G$ does not contain an odd anti-hole.
 Since the complement of a cycle $C_5$ is $C_5$ again, $G$ does not contain an anti-hole on $5$ vertices.
  Note that any odd anti-hole with at least $7$ vertices contains a triangle. Yet, since $\omega(G) \le 2$, $G$ does not contain a triangle. Therefore $G$ does not contain any odd anti-hole with at least $7$ vertices and so we have shown that $G$ is odd anti-hole-free.

  Obviously $\mbox{(iii)}\Rightarrow\mbox{(iv)}$.
  To complete the proof, we will show $\mbox{(iv)}\Rightarrow\mbox{(ii)}$.
 Suppose that $G$ is a weakly perfect graph.
 Then $\chi(G) = \omega(G)$ by definition.
 By Theorem~\ref{thm15}, $\omega(G) \le 3$ and the equality holds if and only if $G$ is chordal.
 If $\omega(G) = 3$, then $G$ is chordal and so, by Theorem~\ref{thm22}, $G$ is a perfect graph. Since we have shown $\mbox{(iii)}\Rightarrow\mbox{(ii)}$, $G$ is odd-hole-free if $\omega(G) = 3$.
 Suppose that $\omega(G)=2$. Then $\chi(G)=2$, so $G$ is a bipartite graph, which implies that $G$ is odd-hole-free.   \end{proof}

 \begin{theorem} {\rm (Weak Perfect Graph Theorem \cite{Lov})} \label{lem:WPGT} A graph is perfect if and only if its complement is perfect.
 \end{theorem}

By definition, it is easy to check that the complement of a Toeplitz graph $G_n\langle s_1, s_2\rangle$ is $G_n\langle t_1, t_2,\ldots,t_{n-3}\rangle$ where $\{t_1,\ldots,t_{n-3}\} = [n-1]\setminus \{s_1,s_2\}$.
Thus, by Theorems~\ref{theorem:perfectgraph} and \ref{lem:WPGT}, the following corollary is true.

 \begin{corollary} Let $G=G_n\langle t_1,\ldots,t_{n-3}\rangle$ and $\{s_1,s_2\} = [n-1]\setminus \{t_1, \ldots, t_{n-3}\}$ with $s_1+s_2 \le n$. Then $G$ is a perfect graph if and only if $(s_1+s_2)/\gcd(s_1,s_2)$ is even or $(s_1+s_2)/\gcd(s_1,s_2)=3$.
  \end{corollary}

 \section{Degree sequence of Toeplitz graphs}
\label{sec:degree}
The {\it degree sequence} of a graph is defined to be a monotonic nonincreasing sequence of the vertex degrees of its graph vertices.
  In this section, we characterize the degree sequence $(d_1,d_2,\ldots,d_n)$ of a Toeplitz graph with $n$ vertices and show that a Toeplitz graph is a regular graph if and only if it is a circulant graph.

  We recall that $B(G) = \{t_1,\ldots,t_k\}$ for $G:=G_n\langle t_1, \ldots, t_k\rangle$.
  For a Toeplitz graph $G = G_n \langle t_1,\ldots, t_k\rangle$, we denote by $\ell_G(i)$ the number of elements in $\{t_1,\ldots,t_k\}$ which are less than $i$.
Then it is easy to see that \begin{equation}\label{eq:ellj} \ell_G(j+1)-\ell_G(j) = 1 \Leftrightarrow j \in B(G)\end{equation}
for any $j \in [n-1]$.

\begin{theorem}\label{thm:ineqdeg}
 Let $G = G_n \langle t_1,\ldots, t_k\rangle$. Then for each $i \in [n]$, \[
 \deg(i) = \ell_G(i) + \ell_G(n-i+1).\] \end{theorem}
 \begin{proof}
  Take a vertex $i \in [n]$. Then \begin{align*}
  \deg(i) &= | \{i-t_1,\ldots, i-t_k,i+t_1,\ldots,i+t_k\} \cap [n] | \\
         &= |\{i-t_1,\ldots, i-t_k\} \cap [n]| + |\{i+t_1,\ldots,i+t_k\} \cap [n]|. \end{align*}
  By definition of $\ell_G(i)$, \[\{i-t_1,\ldots,i-t_k\} \cap [n] = \{i-t_1,\ldots,i-t_{\ell_G(i)}\}\]
   and \[
   \{i+t_1,\ldots,i+t_k\} \cap [n] = \{i+t_1,\ldots,i+t_{\ell_G(n-i+1)}\},\]
  where the right sides of the first and second equalities are $\emptyset$ if $\ell_G(i) = 0$ or $\ell_G(n-i+1)=0$, respectively.
 \end{proof}

\begin{corollary}\label{cor:deg1}
For a Toeplitz graph on $n$ vertices, the following are true:
\begin{itemize}
\item[(a)] For each  $j \in [n]$, $\deg(j)= \deg(n-j+1)$.
\item[(b)] If $n$ is odd, then $\deg\left({\frac{n+1}{2}}\right)$ is even.
\end{itemize}
\end{corollary}

\begin{lemma}\label{lem:deg}
The difference of degrees of two consecutive vertices of a Toeplitz graph is at most $1$. Moreover,  for  $G:=G_n\langle t_1, \ldots, t_k\rangle$ and each vertex $j \in [n-1]$, the following are true: \begin{itemize}
 \item[(a)] $\deg(j)=\deg(j+1)$ if and only if $\{j, n-j\}\subseteq B(G)$ or $\{j, n-j\}\subseteq [n-1]\backslash  B(G)$.
  \item[(b)] $\deg(j)+1=\deg(j+1)$ if and only if $j\in B(G)$ and  $n-j \notin B(G)$.
  \item[(c)] $\deg(j)-1= \deg(j+1)$ if and only if $n-j\in B(G)$ and $j \notin B(G)$.
  \end{itemize}
\end{lemma}
\begin{proof}
Let $G = G_n \langle t_1,\ldots, t_k\rangle$. Take $j \in [n-1]$.
By definition,
\[
0 \le \ell_G(j+1)-\ell_G(j) \le 1 \ \ \mbox{and} \ \ -1 \le \ell_G(n-j) - \ell_G(n-j+1) \le 0.\]
Therefore \[
-1 \le \deg(j+1) - \deg(j) \le 1\]
 by Theorem~\ref{thm:ineqdeg}.
 To show the `moreover' part, suppose that $\deg(j+1) = \deg(j)$. Then either (i) $\ell_G(j+1)-\ell_G(j) = 1$ and $\ell_G(n-j) - \ell_G(n-j+1) = -1$ or (ii) $\ell_G(j+1)-\ell_G(j) = \ell_G(n-j) - \ell_G(n-j+1) = 0$.
Thus, by \eqref{eq:ellj}, $\{j, n-j\}\subseteq B(G)$ in the case (i) and $\{j, n-j\}\subseteq [n-1]\backslash  B(G)$ in the case (ii).
One may check (b) and (c) in a similar manner as above.
\end{proof}

From Corollary~\ref{cor:deg1}, we know that the degree sequence of a Toeplitz graph of order $n$ has the property that each term appears an even number of times, except the degree of $(n+1)/2$ which happens to be even when $n$ is odd.
In addition, the terms form consecutive integers with possible repetition by Lemma~\ref{lem:deg}.
However, this necessary condition cannot be sufficient. To see why, we take the sequence ${\bf d} = (4,3,3,2,2,1,1)$, which satisfies the necessary condition. Suppose that ${\bf d}$ is the degree sequence of a Toeplitz graph $G$. Then $\deg(1) = 1$, $\deg(2) = 2$, $\deg(3) = 3$. By Lemma~\ref{lem:deg}(b), $1 \in B(G)$ and $2 \in B(G)$. Then $1$ is adjacent to $2$ and $3$ and we reach a contradiction.{\footnote{The authors thank Homoon Ryu for finding this counterexample.}}
This counterexample is a special case of $(2m,2m-1,2m-1,2m-2,2m-2,\ldots,2,2,1,1)$ for some $m \ge 2$, which cannot be the degree sequence of any Toeplitz graph of order $4m-1$.

 The following theorem gives a necessary and sufficient condition  for a monotone nonincreasing sequence ${\bf d}_n=(d_1, d_2, \ldots, d_n)$ of nonnegative integers being the degree sequence of a Toeplitz graph.

 \begin{theorem}\label{degree:Toep}
 A monotone nonincreasing sequence ${\bf d}_n=(d_1, d_2, \ldots, d_n)$ of nonnegative integers
  is the degree sequence of a Toeplitz graph if and only if there is a permutation $\pi$ on $[n]$ such that
   \begin{itemize}
  \item[(a)] $|d_{\pi(i+1)}-d_{\pi(i)}| \le 1$  for each $i \in [n-1]$;
  \item[(b)] $d_{\pi(i)} = d_{\pi(n-i+1)}$ for each $i \in [n]$;
  \item[(c)] $s\le d_{\pi(1)} \le n-1-s$;
  \item[(d)] $d_{\pi(1)}$ and $s$ have the same parity if $n$ is odd,
    \end{itemize}
    where $s$ is the number of $i \in \{1,\ldots,\left\lfloor (n-1)/2 \right\rfloor\}$ for which $d_{\pi(i+1)}-d_{\pi(i)} \ne 0$. 
 \end{theorem}
 \begin{proof}
 For notational convenience, we let $m=\left\lfloor (n-1)/2 \right\rfloor$.
  We first show the `only if' part. Suppose that ${\bf d}_n=(d_1, d_2, \ldots, d_n)$ is the degree sequence of a Toeplitz graph $G:=G_n\langle t_1, \ldots, t_k\rangle$.
 Let $\pi$ be the permutation on $[n]$ such that $d_{\pi(i)} = \deg(i)$ for each $i \in [n]$.
Then the conditions (a) and (b) immediately come from Corollary~\ref{cor:deg1} and Lemma~\ref{lem:deg}.
 Obviously, $d_{\pi(1)} =\deg(1) =k$. Now, for each  $i \in [m]$ such that $\deg(i+1) - \deg(i) \neq 0$, exactly one of $i$ or $n-i$ belongs to  $B(G)$ and the other belongs to $[n-1]\setminus B(G)$ by Lemma~\ref{lem:deg} (b) and (c).
  Therefore $s \le |B(G)| =k$ and $s \le |[n-1]\setminus  B(G)| = n-1-k$, so the sequence with the permutation $\pi$ satisfies the condition (c).
Let \[
B^+ = \{ i \in [m] \mid \deg(i+1)-\deg(i) = 1\}\] 
and\[B^- = \{ i \in [m] \mid \deg(i+1)-\deg(i) = -1\}.\]
 Then\[
 d_{\pi(m+1)} -d_{\pi(1)}= \left(d_{\pi(m+1)} - d_{\pi(m)}\right) + \cdots + \left(d_{\pi(2)} - d_{\pi(1)}\right)=|B^+|-|B^-|.\]
Since the parities of $s=|B^+|+|B^-|$ and $|B^+|-|B^-|$ are the same, $d_{\pi(m+1)}$ is odd and we reach a contradiction. Thus the sequence with the permutation $\pi$ satisfies the condition (d).

 Now we show the `if' part.
  Suppose that there exists a permutation $\pi$ on $[n]$ satisfying the conditions (a), (b), (c) and (d).
  For notational convenience, let $k=d_{\pi(1)}$.
 We will construct a Toeplitz graph $G:=G_n\langle t_1, \ldots, t_{k}\rangle$ such that  $\deg(i) = d_{\pi(i)}$ for each $i \in [n]$ as follows.

   Let \[B_1 = \{ i \mid d_{\pi(i+1)} - d_{\pi(i)} = 1, i \in [m] \},\] \[B_2 = \{ i \mid d_{\pi(i+1)} - d_{\pi(i)} = -1, i\in [m] \},\] and\[B_3 = \{i \mid d_{\pi(i+1)} - d_{\pi(i)} = 0, i \in [m]\}.\]  Then $|B_1|+|B_2|+|B_3| = m$ by condition (a) and they are mutually disjoint sets. By definition, $|B_1| + |B_2| =s$ and so $|B_3| = m -s$.
In addition, since $d_{\pi(1)} \le n-1-s$ by the condition (c), $d_{\pi(1)} - s \le n-1 -2s$.
Thus\[
 \left\lfloor \frac{d_{\pi(1)} - s}{2} \right\rfloor \le \left\lfloor \frac{n-1 -2s}{2} \right\rfloor \le |B_3|,\]
 and so we may choose $\left\lfloor(d_{\pi(1)}-s)/2 \right\rfloor$ elements from $B_3$. We denote by $B_3^*$  the set of those $\left\lfloor (d_{\pi(1)}-s)/2 \right\rfloor$ elements.
Now we define $\{t_1,\ldots,t_k\}$ by
   \[
  \begin{cases} B_1 \cup \{n-i\mid i\in B_2\} \cup B_3^* \cup \{n-i \mid i \in B_3^*\} & \mbox{if $d_{\pi(1)}-s$ is even;} \\
    B_1 \cup  \{n-i\mid i\in B_2\} \cup B_3^* \cup \{n-i \mid i \in B_3^*\} \cup\{\frac{n}{2}\} & \mbox{otherwise.}
   \end{cases}
   \]
  Then the set $\{t_1,\ldots,t_k\}$ is well-defined because $\frac{n}{2}$ is integer when $d_{\pi(1)} -s$ is odd,  by the condition (d), and it is straightforward to check that the set defined above has $k$ element.
  
  Let $G=G_n\langle t_1, \ldots, t_{k}\rangle$.
Then \begin{equation} \deg_G(1)=k=d_{\pi(1)}.\label{eq:d1}\end{equation}

Take $i \in [m]$. By the definition of $B(G)$, $i\in B_1$ if $i \in B(G)$ and $n-i \notin B(G)$;
 $i \in B_2$ if $i \notin B(G)$ and $n-i \in B(G)$; $i \in B_3^*$ if either $i \in B(G)$ and $n-i \in B(G)$ or $i \notin B(G)$ and $n-i \notin B(G)$.
Thus $\deg(i) = d_{\pi(i)}$ by Lemma~\ref{lem:deg}.  By the way, for each  $i \in [m]$, $d_{\pi(i)} = d_{\pi(n-i+1)}$  by the condition (b), so $\deg(n-i+1) = \deg(i) =d_{\pi(i)} = d_{\pi(n-i+1)}$.
Hence $d_{\pi(i)} = \deg(i)$ for each $i \in [n]$ and ${\bf d}_n$ is the degree sequence of $G$.\end{proof}

In the rest of this section, we will show that a necessary and sufficient condition for a Toeplitz graph being regular is its being a circulant graph. To do so, we need the following lemma.
 \begin{lemma}\label{lem:cir} Let $G=G_n\langle t_1, t_2,\ldots, t_k\rangle$. Then $G$ is a circulant graph if and only if $t_{k+1-i}=n-t_i$ for each $i \in [k]$.\end{lemma}
 \begin{proof} Let $(a_{ij})_{1 \le i,j \le n}$ be the adjacency matrix of $G$.
 To show the `only if' part, suppose that $G$ is a circulant graph.

 Take $i \in [n-1]$. Suppose that $i \in B(G)$ (recall $B(G) = \{t_1,\ldots,t_k\}$).  Then $a_{i+1 , 1} =1$ by definition.  By \eqref{eq:circulant}, $a_{n-i+1, 1}=a_{i+1,1}=1$. Therefore $n-i \in B(G)$.
 Thus \[
 B(G) = \{n-t_k,\ldots, n-t_1\}.\]
  Since $t_1<t_2<\cdots<t_k$, $t_{k+1-j} = n-t_j$ for each $j \in [k]$.

  Now we show the `if' part. Suppose that $t_{k+1-i}=n-t_i$ for each $i \in [n-1]$.
  Then, for each $s \in [n-1]$, $a_{s+1,1} = 1$ if and only if $s \in B(G)$ if and only if $n-s \in B(G)$ if and only if $a_{n-s+1,1} = 1$. Thus $(a_{ij})_{1 \le i,j \le n}$ is a circulant matrix.
 \end{proof}

Let $G$ be a circulant graph isomorphic to $G_n\langle t_1, \ldots, t_k\rangle$.
 It is well-known that if $t_i \neq n/2$ for each $i \in [k]$, then $k$ is even and $G$ is a regular graph; if $n$ is even and $t_j= n/2$ for some $j \in [k]$, then $k$ is odd and $G$ is a regular graph (see \cite{Mei}). Therefore any circulant graph is a regular Toeplitz graph. In the following, we show that its converse is also true.

 \begin{theorem} \label{thm16} A Toeplitz graph is regular if and only if it is a circulant graph.
\end{theorem}
 \begin{proof} By the argument above, it is sufficient to show the `only if' part.
  Let $G=G_n\langle t_1, \ldots, t_k\rangle$ be a regular Toeplitz graph. Then  $\deg(i)=\deg(i+1)$ for each $i \in \{1,\ldots, \left\lceil n/2 \right\rceil-1\}$.
  Now, by Lemma~\ref{lem:deg}(a), either $\{i, n-i\} \subset B(G)$ or $\{i,n-i\} \subset [n-1]\backslash  B(G)$  for each $i \in [n-1]$. Since $B(G)\ne \emptyset$, \[
  B(G) = \{t_1,\ldots,t_k\} = \{n-t_k,\ldots,n- t_1\}.\]
 Thus $t_{k+1-i}=n-t_i$ for each $i \in [k]$ and so, by Lemma~\ref{lem:cir}, $G$ is a circulant graph.
 \end{proof}

 \section{Closing remarks}
 \label{sec:open}
 Theorem~\ref{thm:chordequiv} gives equivalent conditions for a Toeplitz graph $G_n\langle t_1, t_2,\ldots, t_k\rangle$ with $t_1<\cdots<t_k$ and $n \ge t_{k-1}+t_{k}$ being chordal.
 In the proof of Theorem~\ref{thm:chordequiv}, we did not use the condition $n \ge t_{k-1}+t_{k}$ when we showed $\mbox{(iii)}\Leftrightarrow\mbox{(iv)}$, $\mbox{(iii)}\Rightarrow \mbox{(i)}$, and $\mbox{(i)}\Rightarrow \mbox{(ii)}$.

 \smallskip\noindent
  {\bf Question 1.} Does the statement (ii) imply the statement (i) in Theorem~\ref{thm:chordequiv} without the condition $n \ge t_{k-1}+t_{k}$?

 \smallskip
 Theorem~\ref{theorem:perfectgraph} gives equivalent conditions for a Toeplitz graph $G_n\langle t_1, t_2\rangle$ being perfect for any positive integers $n$ and $t_1 < t_2$.

 \smallskip\noindent
{\bf Question 2.} Is there any meaningful characterization for (weakly) perfect Toeplitz graphs $G_n\langle t_1, t_2,\ldots, t_k\rangle$ with $k \ge 3$?

\section*{Acknowledgements}
The authors thank anonymous referees for his/her helpful comments to improve the paper (in particular, the proof of Lemma~\ref{lem:k2cycle} and Section~\ref{sec:degree}).
This work was supported by the National Research Foundation of Korea(NRF) grant funded by the Korea government(MSIP) (2016R1A5A1008055, NRF-2017R1E1A1A03070489, 2021R1C1C2014187), the Ministry of Education (NRF-2016R1A6A3A11930452).

\end{document}